\documentclass{amsart}
\usepackage{amsmath}
\usepackage{amssymb}
\usepackage{amsthm}
\usepackage{epsfig}
\usepackage{amsfonts}
\usepackage{amscd}
\usepackage{mathrsfs}
\usepackage{enumerate}
\usepackage{latexsym}

\newtheorem{theorem}{Theorem}[section]
\newtheorem{lemma}[theorem]{Lemma}

\theoremstyle{definition}
\newtheorem{definition}[theorem]{Definition}

\newtheorem*{xclaim}{Claim}

\theoremstyle{remark}
\newtheorem{remark}[theorem]{Remark}
\numberwithin{equation}{section}

\begin{document}

\title[Representations of certain Banach algebras]{Representations of certain Banach algebras}

\author{M.R. Koushesh}

\address{Department of Mathematical Sciences, Isfahan University of Technology, Isfahan 84156--83111, Iran}
\address{School of Mathematics, Institute for Research in Fundamental Sciences (IPM), P.O. Box: 19395--5746, Tehran, Iran}
\email{koushesh@cc.iut.ac.ir}
\thanks{This research was in part supported by a grant from IPM (No. 90030052).}




\begin{abstract}
For a space $X$ denote by $C_b(X)$ the Banach algebra of all continuous bounded scalar-valued functions on $X$ and denote by $C_0(X)$ the set of all elements in $C_b(X)$ which vanish at infinity.

We prove that certain Banach subalgebras $H$ of $C_b(X)$ are isometrically isomorphic to $C_0(Y)$, for some unique (up to homeomorphism) locally compact Hausdorff space $Y$. The space $Y$ is explicitly constructed as a subspace of the Stone--\v{C}ech compactification of $X$. The known construction of $Y$ enables us to examine certain properties of either $H$ or $Y$ and derive results not expected to be deducible from the standard Gelfand Theory.
\end{abstract}

\maketitle


\section{Introduction}

By a {\em space} we mean a {\em topological space}; completely regular spaces are Hausdorff. Throughout this article the underlying field of scalars (which is fixed throughout each discussion) is assumed to be either the real field $\mathbb{R}$ or the complex field $\mathbb{C}$, unless specifically stated otherwise.

Let $X$ be a completely regular space. We denote by $C_b(X)$ the set of all continuous bounded scalar-valued functions on $X$. If $f\in C_b(X)$, the {\em zero-set} of $f$, denoted by $\mathrm{Z}(f)$, is $f^{-1}(0)$, the {\em cozero-set} of $f$, denoted by $\mathrm{Coz}(f)$, is $X\backslash \mathrm{Z}(f)$, and the {\em support} of $f$, denoted by $\mathrm{supp}(f)$, is $\mathrm{cl}_X\mathrm{Coz}(f)$. Let
\[\mathrm{Coz}(X)=\big\{\mathrm{Coz}(f):f\in C_b(X)\big\}.\]
The elements of $\mathrm{Coz}(X)$ are called {\em cozero-sets} of $X$. Denote by $C_0(X)$ the set of all $f\in C_b(X)$ which vanish at infinity (that is, $|f|^{-1}([\epsilon,\infty))$ is compact for each $\epsilon>0$) and denote by $C_{00}(X)$ the set of all $f\in C_b(X)$ with compact support.

This work is a continuation (though it is self-contained) of our previous work \cite{Ko6} (and \cite{Ko10}; see also the follow-up preprint \cite{Ko12}) in which we have studied the Banach algebra of continuous bounded scalar-valued functions with separable support defined on a locally separable metrizable space $X$. Indeed, this article is an outgrowth of author's unsuccessful (partly successful, however) attempt to give an alternative proof of the celebrated commutative Gelfand--Naimark Theorem. We show that certain Banach subalgebras $H$ of $C_b(X)$ are representable as $C_0(Y)$ for some unique locally compact Hausdorff space $Y$. We construct $Y$ explicitly as a subspace of the Stone--\v{C}ech compactification of $X$. The known construction of $Y$ enables us to examine certain properties of either $H$ or $Y$ and derive  results not expected to be deducible from the standard Gelfand Theory.

We now briefly review certain facts from General Topology. Additional information may be found in \cite{E} and \cite{GJ}.

\subsection*{1.1. The Stone--\v{C}ech compactification.} Let $X$ be a completely regular space. By a {\em compactification} $\gamma X$ of $X$ we mean a compact Hausdorff space $\gamma X$ containing $X$ as a dense subspace. The {\em Stone--\v{C}ech compactification} $\beta X$ of $X$ is the compactification of $X$ which is characterized among all compactifications of $X$ by the following property: Every continuous $f:X\rightarrow K$, where $K$ is a compact Hausdorff space, is continuously extendable over $\beta X$; denote by $f_\beta$ this continuous extension of $f$. The Stone--\v{C}ech compactification of a completely regular space always exists. In what follows use will be made of the following properties of $\beta X$.
\begin{itemize}
  \item $X$ is locally compact if and only if $X$ is open in $\beta X$.
  \item Any open-closed subspace of $X$ has open-closed closure in $\beta X$.
  \item If $X\subseteq T\subseteq\beta X$ then $\beta T=\beta X$.
  \item If $X$ is normal then $\beta T=\mathrm{cl}_{\beta X}T$ for any closed subspace $T$ of $X$.
\end{itemize}

\subsection*{1.2. Locally-$\mathscr{P}$ spaces.} Let $\mathscr{P}$ be a topological property. A space $X$ is called {\em locally-$\mathscr{P}$}, if each $x\in X$ has an open neighborhood $U$ in $X$ whose closure $\mathrm{cl}_XU$ has $\mathscr{P}$.

\subsection*{1.3. Metrizable spaces and separability.} The {\em density} of a space $X$, denoted by $d(X)$, is defined by
\[d(X)=\min\big\{|D|:D\mbox{ is dense in }X\big\}+\aleph_0.\]
In particular, a space $X$ is separable if and only if $d(X)=\aleph_0$. Note that in any metrizable space the three notions of separability, being Lindel\"{o}f, and second countability coincide; thus any subspace of a separable metrizable space is separable. By a theorem of Alexandroff, any locally separable metrizable space $X$ can be represented as a disjoint union
\[X=\bigcup_{i\in I}X_i,\]
where $I$ is an index set, and $X_i$ is a non-empty separable open-closed subspace of $X$ for each $i\in I$. (See Problem 4.4.F of \cite{E}.) Note that $d(X)=|I|$ if $I$ is infinite.

\subsection*{1.4. Paracompact spaces and the Lindel\"{o}f property.} Let $X$ a regular space. For any open covers $\mathscr{U}$ and $\mathscr{V}$ of $X$ we say that $\mathscr{U}$ is a {\em refinement} of $\mathscr{V}$ if every element of $\mathscr{U}$ is contained in an element of $\mathscr{V}$. An open cover $\mathscr{U}$ of $X$ is called {\em locally finite} if each point of $X$ has an open neighborhood in $X$ intersecting only a finite number of the elements of $\mathscr{U}$. The space $X$ is called {\em paracompact} if for every open cover $\mathscr{U}$ of $X$ there is an open cover of $X$ which refines $\mathscr{U}$. Every metrizable space is paracompact and every paracompact space is normal. The {\em Lindel\"{o}f number} of $X$, denoted by $\ell(X)$, is defined by
\[\ell(X)=\min\{\mathfrak{n}:\mbox{any open cover of $X$ has a subcover of cardinality}\leq\mathfrak{n}\}+\aleph_0.\]
In particular, $X$ is Lindel\"{o}f if and only if $\ell(X)=\aleph_0$. Any locally compact paracompact space $X$ can be represented as a disjoint union
\[X=\bigcup_{i\in I}X_i,\]
where $I$ is an index set, and $X_i$ is a non-empty Lindel\"{o}f open-closed subspace of $X$ for each $i\in I$. (See Theorem 5.1.27 of \cite{E}.) Note that $\ell(X)=|I|$ if $I$ is infinite.

\section{The representation theorem}

The subspace $\lambda_H X$ of $\beta X$ defined below plays a crucial role in our study.

\begin{definition}\label{HGA}
Let $X$ be a completely regular space and let $H\subseteq C_b(X)$. Define
\[\lambda_H X=\bigcup\big\{\mathrm{int}_{\beta X}\mathrm{cl}_{\beta X}\mathrm{Coz}(h):h\in H\big\}.\]
\end{definition}

The above definition of $\lambda_H X$ is motivated by the definition of $\lambda_{\mathscr P} X$ (here ${\mathscr P}$ is a topological property) as given in \cite{Ko3} (also, in \cite{Ko4}, \cite{Ko5}, \cite{Ko13} and \cite{Ko11}). Note that $\lambda_H X$ is open in $\beta X$ and is thus locally compact.

Recall that if $X$ is a space and $D$ is a dense subspace of $X$, then
\[\mathrm{cl}_XU=\mathrm{cl}_X(U\cap D)\]
for any open subspace $U$ of $X$.

\begin{lemma}\label{KHF}
Let $X$ be a completely regular space and let $H\subseteq C_b(X)$ such that
\begin{itemize}
\item[\rm$(*)$] For any $x\in X$ there is some $h\in H$ with $h(x)\neq0$.
\end{itemize}
Then
\[X\subseteq\lambda_H X.\]
\end{lemma}

\begin{proof}
Let $x\in X$. By $(*)$ there is some $h\in H$ with $h(x)\neq0$. Since
\[\mathrm{Coz}(h_\beta)\subseteq\mathrm{cl}_{\beta X}\mathrm{Coz}(h_\beta)=\mathrm{cl}_{\beta X}\big(X\cap \mathrm{Coz}(h_\beta)\big)=\mathrm{cl}_{\beta X}\mathrm{Coz}(h)\]
we have
\[x\in \mathrm{Coz}(h_\beta)\subseteq\mathrm{int}_{\beta X}\mathrm{cl}_{\beta X}\mathrm{Coz}(h)\subseteq\lambda_H X.\]
\end{proof}

A version of the classical Banach--Stone Theorem states that if $X$ and $Y$ are locally compact Hausdorff spaces, the  Banach algebras (or even the rings; see \cite{A}) $C_0(X)$ and $C_0(Y)$ are isometrically isomorphic if and only if the spaces $X$ and $Y$ are homeomorphic (see Theorem 7.1 of \cite{Be}); this will be used in the proof of the following theorem.

\begin{theorem}\label{TRES}
Let $X$ be a completely regular space. Let $H$ be a Banach subalgebra of $C_b(X)$ such that
\begin{itemize}
\item[\rm(1)] For any $x\in X$ there is some $h\in H$ with $h(x)\neq0$.
\item[\rm(2)] For any $f\in C_b(X)$, if $\mathrm{supp}(f)\subseteq\mathrm{supp}(h)$ for some $h\in H$, then $f\in H$.
\end{itemize}
Then $H$ is isometrically isomorphic to $C_0(Y)$ for some unique locally compact Hausdorff space $Y$, namely $Y=\lambda_H X$. Furthermore, the following are equivalent:
\begin{itemize}
\item[\rm(a)] $H$ is unital.
\item[\rm(b)] $H$ contains $\mathbf{1}$.
\item[\rm(c)] $Y$ is compact.
\item[\rm(d)] $Y=\beta X$.
\end{itemize}
\end{theorem}

\begin{proof}
For any $f\in C_b(X)$ denote
\[f_H=f_\beta|\lambda_H X.\]
By Lemma \ref{KHF} we know that $f_H$ extends $f$. Note that $\mathrm{supp}(|f|)=\mathrm{supp}(f)$ for any $f\in C_b(X)$; thus by (2) we have $|f|\in H$ if $f\in H$.

We divide the proof into verification of several claims.

\begin{xclaim}\label{HGA}
\textit{For any $f\in C_b(X)$ the following are equivalent:
\begin{itemize}
\item[\rm(i)] $f\in H$.
\item[\rm(ii)] $f_H\in C_0(\lambda_H X)$.
\end{itemize}}
\end{xclaim}

\noindent {\em Proof of the claim.} (i) {\em  implies} (ii). Note that $\mathrm{Coz}(f_\beta)\subseteq\lambda_H X$, as
\[\mathrm{int}_{\beta X}\mathrm{cl}_{\beta X}\mathrm{Coz}(f)\subseteq\lambda_H X\]
and
\[\mathrm{Coz}(f_\beta)\subseteq\mathrm{int}_{\beta X}\mathrm{cl}_{\beta X}\mathrm{Coz}(f),\]
since
\[\mathrm{Coz}(f_\beta)\subseteq\mathrm{cl}_{\beta X}\mathrm{Coz}(f_\beta)=\mathrm{cl}_{\beta X}\big(X\cap\mathrm{Coz}(f_\beta)\big)=\mathrm{cl}_{\beta X}\mathrm{Coz}(f).\]
If $\epsilon>0$ then
\[|f_H|^{-1}\big([\epsilon,\infty)\big)=|f_\beta|^{-1}\big([\epsilon,\infty)\big)\]
is closed in $\beta X$ and is therefore compact.

(ii) {\em  implies} (i). Let $n$ be a positive integer. Since $|f_H|^{-1}([1/n,\infty))$ is a compact subspace of $\lambda_H X$, we have
\begin{eqnarray*}
|f_H|^{-1}\big([1/n,\infty)\big)&\subseteq&\bigcup_{i=1}^{k_n}\mathrm{int}_{\beta X}\mathrm{cl}_{\beta X}\mathrm{Coz}(h_i)\\&\subseteq&\bigcup_{i=1}^{k_n}\mathrm{cl}_{\beta X}\mathrm{Coz}(h_i)\\&=&\mathrm{cl}_{\beta X}\Big(\bigcup_{i=1}^{k_n}\mathrm{Coz}(h_i)\Big)\\&=&\mathrm{cl}_{\beta X}\Big(\bigcup_{i=1}^{k_n} \mathrm{Coz}\big(|h_i|\big)\Big)=\mathrm{cl}_{\beta X}\mathrm{Coz}\Big(\sum_{i=1}^{k_n} |h_i|\Big)
\end{eqnarray*}
for some $h_1,\ldots,h_{k_n}\in H$. Let
\[g_n=|h_1|+\cdots+|h_{k_n}|\in H.\]
We have
\begin{eqnarray*}
|f|^{-1}\big([1/n,\infty)\big)&=&X\cap|f_H|^{-1}\big([1/n,\infty)\big)\\&\subseteq& X\cap\mathrm{cl}_{\beta X}\mathrm{Coz}(g_n)=\mathrm{cl}_X\mathrm{Coz}(g_n)=\mathrm{supp}(g_n).
\end{eqnarray*}
Let
\[g=\sum_{n=1}^\infty 2^{-n}\frac{g_n}{\|g_n\|}.\]
(We may assume that $g_n\neq\mathbf{0}$ for each positive integer $n$.) Then $g\in H$. Since
\[\mathrm{Coz}(f)=\bigcup_{n=1}^\infty|f|^{-1}\big([1/n,\infty)\big)\subseteq\bigcup_{n=1}^\infty\mathrm{supp}(g_n)\subseteq\mathrm{supp}(g)\]
we have $\mathrm{supp}(f)\subseteq\mathrm{supp}(g)$, which by (2) implies that $f\in H$. This proves the claim.

\begin{xclaim}\label{GHJ}
\textit{Let
\[\psi:H\rightarrow C_0(\lambda_H X)\]
be defined by $\psi(h)=h_H$ for any $h\in H$. Then $\psi$ is an isometric isomorphism.}
\end{xclaim}

\noindent {\em Proof of the claim.} By the first claim the function $\psi$ is well-defined. It is clear that $\psi$ is a homomorphism and that $\psi$ is injective. (Note that $X\subseteq\lambda_H X$ by Lemma \ref{KHF}, and that any two scalar-valued continuous functions on $\lambda_H X$ coincide, provided that they agree on $X$.)
To show that $\psi$ is surjective, let $g\in C_0(\lambda_H X)$. Then $(g|X)_H=g$ and thus $g|X\in H$ by the first claim. Now $\psi(g|X)=g$. Finally, to show that $\psi$ is an isometry, let $h\in H$. Then
\[|h_H|(\lambda_H X)=|h_H|(\mathrm{cl}_{\lambda_H X}X)\subseteq\mathrm{cl}_{\mathbb{R}}\big(|h_H|(X)\big)=\mathrm{cl}_{\mathbb{R}}\big(|h|(X)\big)\subseteq\big[0,\|h\|\big]\]
which yields $\|h_H\|\leq\|h\|$. That $\|h\|\leq\|h_H\|$ is clear, as $h_H$ extends $h$. This proves the claim.

\medskip

The uniqueness part of the theorem follows from the Banach--Stone Theorem. (Note that $\lambda_H X$ is a locally compact Hausdorff space.)

We now verify the final assertion of the theorem. Suppose that $H$ is unital with the unit element $u$. For each $x\in X$ let $h_x\in H$ such that $h_x(x)\neq 0$ ($h_x$ exists by (1)). Then $u(x) h_x(x)=h_x(x)$ which yields $u(x)=1$. That is $u=\mathbf{1}$. Note that $\mathrm{Coz}(u)=X$. Now, by the way $\lambda_H X$ is defined we have $\lambda_H X=\beta X$. For the converse, note that if $Y$ is compact then $C_0(Y)=C_b(Y)$. Thus $H$ is unital, as it is isometrically isomorphic to $C_0(Y)$ and the latter is so.
\end{proof}

\begin{remark}\label{KGKLF}
In Theorem \ref{TRES} the existence of the locally compact space $Y$ may also be deduced from the commutative Gelfand--Naimark Theorem in which $Y$ would be the structure space or the maximal ideal space of $H$. (See \cite{HR}.) The point in Theorem \ref{TRES} is that we construct $Y$ explicitly as a subspace of the Stone--\v{C}ech compactification of $X$. The known construction of $Y$ will then enable us to derive certain of its properties. This will be illustrated in Section \ref{KKCDXF} where we consider special cases of the Banach subalgebra $H$.
\end{remark}

\begin{remark}\label{KGFF}
Condition (2) in Theorem \ref{TRES} can be replaced by the following rather more abstract condition.
\begin{itemize}
\item[\rm(2)$'$] For any $f\in C_b(X)$, if $\mathrm{ann}(h)\subseteq\mathrm{ann}(f)$ for some $h\in H$, then $f\in H$.
\end{itemize}
Here
\[\mathrm{ann}(f)=\big\{g\in C_b(X): fg=\mathbf{0}\big\}\]
for any $f\in C_b(X)$.

To show that (2) (in Theorem \ref{TRES}) implies (2)$'$, let $f\in C_b(X)$ such that $\mathrm{ann}(h)\subseteq\mathrm{ann}(f)$ for some $h\in H$. Suppose to the contrary that $\mathrm{supp}(f)\nsubseteq\mathrm{supp}(h)$. Let $x\in\mathrm{supp}(f)$ with $x\notin\mathrm{supp}(h)$. There exists a continuous $k:X\rightarrow[0,1]$ such that
\[k(x)=1\;\;\;\;\mbox{ and }\;\;\;\;k|\mathrm{supp}(h)=\mathbf{0}.\]
Now $k^{-1}((1/2,1])$ is an open neighborhood of $x$ in $X$. Therefore
\[I=k^{-1}\Big(\Big(\frac{1}{2},1\Big]\Big)\cap\mathrm{Coz}(f)\neq\emptyset.\]
If we let $t\in I$ then $k(t)f(t)\neq0$. But $k\in\mathrm{ann}(h)$, as $kh=\mathbf{0}$ by the way we have defined $k$, and therefore $k\in\mathrm{ann}(f)$, which is a contradiction. Thus $\mathrm{supp}(f)\subseteq\mathrm{supp}(h)$. By (2), this implies that $f\in H$.

Next, we verify that (2)$'$ implies (2). Let $f\in C_b(X)$ such that $\mathrm{supp}(f)\subseteq\mathrm{supp}(h)$ for some $h\in H$. Suppose to the contrary that $\mathrm{ann}(h)\nsubseteq\mathrm{ann}(f)$. Let $g\in\mathrm{ann}(h)$ with $g\notin\mathrm{ann}(f)$. Then $gf\neq\mathbf{0}$. Let $t\in X$ such that $g(t)f(t)\neq0$. Then $r=|g(t)|>0$. Note that $t\in\mathrm{supp}(f)$, as $f(t)\neq0$, and thus $t\in\mathrm{supp}(h)$. Now $|g|^{-1}((r/2,\infty))$ is an open neighborhood of $t$ in $X$. Therefore
\[E=|g|^{-1}\Big(\Big(\frac{r}{2},\infty\Big)\Big)\cap\mathrm{Coz}(h)\neq\emptyset.\]
But if $x\in E$ then $g(x)h(x)\neq 0$, which is a contradiction, as $gh=\mathbf{0}$ by the choice of $g$. Thus $\mathrm{ann}(h)\subseteq\mathrm{ann}(f)$. By (2)$'$, this implies that $f\in H$.
\end{remark}

\begin{remark}\label{KGFF}
In Theorem \ref{TRES} observe that the space $Y$ is compact if and only if $X=\mathrm{supp}(h)$ for some $h\in H$. To show this, suppose that $\lambda_H X$ is compact. Then
\begin{equation}\label{KDJB}
\lambda_H X=\bigcup_{i=1}^{k_n}\mathrm{int}_{\beta X}\mathrm{cl}_{\beta X}\mathrm{Coz}(h_i)\subseteq\bigcup_{i=1}^{k_n}\mathrm{cl}_{\beta X}\mathrm{Coz}(h_i)
\end{equation}
for some $h_1,\ldots,h_{k_n}\in H$. Note that $X\subseteq\lambda_H X$ by Lemma \ref{KHF}. Now, if we intersect both sides of (\ref{KDJB}) with $X$ we obtain
\[X=X\cap\lambda_H X\subseteq\bigcup_{i=1}^{k_n}\big(X\cap\mathrm{cl}_{\beta X}\mathrm{Coz}(h_i)\big)=\bigcup_{i=1}^{k_n}\mathrm{cl}_X\mathrm{Coz}(h_i)=\bigcup_{i=1}^{k_n}\mathrm{cl}_X\mathrm{Coz}\big(|h_i|\big)\]
and thus
\[X=\bigcup_{i=1}^{k_n}\mathrm{cl}_X\mathrm{Coz}\big(|h_i|\big)=\mathrm{cl}_X\Big(\bigcup_{i=1}^{k_n}\mathrm{Coz}\big(|h_i|\big)\Big)=
\mathrm{cl}_X\mathrm{Coz}\Big(\sum_{i=1}^{k_n}|h_i|\Big).\]
That is
\[X=\mathrm{cl}_X\mathrm{Coz}(h)=\mathrm{supp}(h),\]
where
\[h=|h_1|+\cdots+|h_{k_n}|\in H.\]

The converse is trivial, as if $X=\mathrm{supp}(h)$ for some $h\in H$, then $\lambda_H X=\beta X$ by the definition of $\lambda_H X$.
\end{remark}

\begin{remark}
In Theorem \ref{TRES} it is natural to ask what would happen if we start with $H=C_0(X)$ from the outset, with $X$ a locally compact Hausdorff space. Given that $x\in X$, there exists an open neighborhood $U_x$ of $x$ in $X$ with compact closure $\mathrm{cl}_X U_x$. Since $X$ is completely regular, there is a continuous $h_x:X\rightarrow\mathbb{R}$ such that
\[h_x(x)=1\;\;\;\;\mbox{ and }\;\;\;\;h_x|(X\backslash U_x)=\mathbf{0}.\]
Then $\mathrm{supp}(h_x)$ is compact, as $\mathrm{supp}(h_x)\subseteq\mathrm{cl}_X U_x$ (and $\mathrm{cl}_X U_x$ is compact). Thus in particular $h_x\in C_0(X)$. In other words, if $H=C_0(X)$, then condition (1) in Theorem \ref{TRES} holds, though, as we will see now, condition (2) therein may fail.

Let $X=\mathbb{R}$. Define $h:X\rightarrow\mathbb{R}$ by
\[h(x)=\frac{1}{1+x^2}\]
for any $x\in X$. Then obviously $h\in C_0(X)$. Now, if we let $f=\mathbf{1}$ then $\mathrm{supp}(f)\subseteq\mathrm{supp}(h)$ while $f\notin C_0(X)$.

More generally, we show that $C_0(X)$ fails to satisfy condition (2) in Theorem \ref{TRES} for any non-compact locally compact $\sigma$-compact Hausdorff space $X$. Let $X$ be a non-compact locally compact $\sigma$-compact Hausdorff space. Then, there exists a sequence $X_1,X_2,\ldots$ of compact subspaces of $X$ such that
\begin{equation}\label{KHJD}
X=\bigcup_{n=1}^\infty X_n,
\end{equation}
and $X_n\subseteq\mathrm{int}_X X_{n+1}$ for each positive integer $n$. (See Exercise 3.8.C of \cite{E}.) By complete regularity of $X$, for each positive integer $n$ there exists a continuous $h_n:X\rightarrow[0,1]$ such that
\[h_n|X_n=\mathbf{1}\;\;\;\;\mbox{ and }\;\;\;\;h_n|(X\backslash\mathrm{int}_X X_{n+1})=\mathbf{0}.\]
Let
\[h=\sum_{n=1}^\infty\frac{h_n}{2^n}.\]
Then $h:X\rightarrow[0,1]$ is continuous. Fix some positive integer $n$. Let $x\in X\backslash X_{n+2}$. Then $x\notin X_i$ if $1<i\leq n+2$ and thus $h_{i-1}(x)=0$ (by the definition of $h_{i-1}$). Therefore
\[h(x)=\sum_{i=n+2}^\infty\frac{h_i(x)}{2^i}\leq\sum_{i=n+2}^\infty\frac{1}{2^i}=\frac{1}{2^{n+1}}<\frac{1}{2^n}.\]
That is
\[E=h^{-1}\Big(\Big[\frac{1}{2^n},\infty\Big)\Big)\subseteq X_{n+2}.\]
Since $X_{n+2}$ is compact, so is its closed subspace $E$. This shows that $h\in C_0(X)$. Note that $\mathrm{supp}(h)=X$ by (\ref{KHJD}) (and the definitions of $h$ and $h_n$'s). Now, as argued in the case when $X=\mathbb{R}$, it follows that $C_0(X)$ does not satisfy condition (2) in Theorem \ref{TRES}.

As a concluding remark, observe that if $C_0(X)$ satisfies the assumption of Theorem \ref{TRES} where $X$ is a locally compact Hausdorff space, then $C_0(X)$ is isometrically isomorphic to $C_0(\lambda_H X)$ by the theorem. Now (a version of) the Banach--Stone Theorem implies that $X$ and $\lambda_H X$ are homeomorphic spaces.
\end{remark}

\section{Examples}\label{KKCDXF}

In this section we give examples of spaces $X$ and Banach subalgebras $H$ of $C_b(X)$ for which Theorem \ref{TRES} is applicable.

Recall that a topological property $\mathscr{P}$ is called {\em hereditary with respect to closed subspaces}, if each closed subspace of a space with $\mathscr{P}$ has $\mathscr{P}$. We will always assume that a topological property is {\em non-empty}, that is, for a given topological property $\mathscr{P}$ there always exists a space with $\mathscr{P}$. This in particular implies that $\emptyset$ has $\mathscr{P}$ for any topological property $\mathscr{P}$ which is hereditary with respect to closed subspaces.

\begin{theorem}\label{CDXF}
Let $\mathscr{Q}$ be a topological property hereditary with respect to closed subspaces. Let $\mathscr{P}$ be a topological property such that
\begin{itemize}
\item[\rm(1)] $\mathscr{P}$ is hereditary with respect to closed subspaces of spaces with $\mathscr{Q}$.
\item[\rm(2)] Any space with $\mathscr{Q}$ containing a dense subspace with $\mathscr{P}$ has $\mathscr{P}$.
\item[\rm(3)] Any space which is a countable union of its closed subspaces each with $\mathscr{P}$ has $\mathscr{P}$.
\end{itemize}
Let $X$ be a completely regular locally-$\mathscr{P}$ space with $\mathscr{Q}$. Let
\[H=\big\{f\in C_b(X):\mathrm{supp}(f)\mbox{ has }\mathscr{P}\big\}.\]
Then $H$ is a Banach algebra isometrically isomorphic to $C_0(Y)$ for some unique locally compact Hausdorff space $Y$, namely $Y=\lambda_H X$. Furthermore, the following are equivalent:
\begin{itemize}
\item[\rm(a)] $X$ has $\mathscr{P}$.
\item[\rm(b)] $H$ is unital.
\item[\rm(c)] $H$ contains $\mathbf{1}$.
\item[\rm(d)] $Y$ is compact.
\item[\rm(e)] $Y=\beta X$.
\end{itemize}
\end{theorem}

\begin{proof}
We verify that $H$ satisfies the assumption of Theorem \ref{TRES}. First, we need to show that $H$ is a subalgebra of $C_b(X)$. Note that $\mathbf{0}\in H$, as $\emptyset$ has $\mathscr{P}$ by (1). Let $f,g\in H$. Then $\mathrm{supp}(f)\cup\mathrm{supp}(g)$ has $\mathscr{Q}$, as it is closed in $X$, and has $\mathscr{P}$, as it is the union of two of its closed subspaces each with $\mathscr{P}$; see (3). Since
\[\mathrm{supp}(f+g)\subseteq\mathrm{supp}(f)\cup\mathrm{supp}(g)\]
it then follows that $\mathrm{supp}(f+g)$ has $\mathscr{P}$, by (1). That is $f+g\in H$. Analogously, $fg\in H$ and $cf\in H$ for any scalar $c$. (Note that $\mathrm{supp}(fg)\subseteq\mathrm{supp}(f)$.)

Next, we show that $H$ is closed in $C_b(X)$. Let $h_1,h_2,\ldots$ be a sequence in $H$ converging to some $f\in C_b(X)$. Note that
\[E=\mathrm{cl}_X\Big(\bigcup_{n=1}^\infty\mathrm{supp}(h_n)\Big)\]
has $\mathscr{Q}$, as it is closed in $X$, and thus by (2) has $\mathscr{P}$, as it contains $\bigcup_{n=1}^\infty\mathrm{supp}(h_n)$ as a dense subspace, and by (3) the latter has $\mathscr{P}$, as it is a countable union of its closed subspaces each with $\mathscr{P}$. Note that
\[\mathrm{Coz}(f)\subseteq\bigcup_{n=1}^\infty\mathrm{Coz}(h_n).\]
Therefore $\mathrm{supp}(f)$ has $\mathscr{P}$ by (1), as it is closed in $E$. That is $f\in H$. This shows that $H$ is a Banach subalgebra of $C_b(X)$.

Next, we verify that $H$ satisfies conditions (1) and (2) of Theorem \ref{TRES}.

To show that $H$ satisfies condition (1) of Theorem \ref{TRES}, let $x\in X$. Let $U$ be an open neighborhood of $x$ in $X$ such that $\mathrm{cl}_XU$ has $\mathscr{P}$. Note that $\mathrm{cl}_XU$ has also $\mathscr{Q}$, as it is closed in $X$. Let $f:X\rightarrow[0,1]$ be continuous with
\[f(x)=1\;\;\;\;\mbox{ and }\;\;\;\;f|(X\backslash U)=\mathbf{0}.\]
Then $\mathrm{supp}(f)$ has $\mathscr{P}$ by (1), as it is closed in $\mathrm{cl}_XU$. Therefore $f\in H$.

That $H$ satisfies condition (2) in Theorem \ref{TRES} is clear, as if $f\in C_b(X)$ with $\mathrm{supp}(f)\subseteq\mathrm{supp}(h)$ for some $h\in H$, then $\mathrm{supp}(h)$ has $\mathscr{P}$ (and also $\mathscr{Q}$, as it is closed in $X$) and thus, by (1) again, so does its closed subspace $\mathrm{supp}(f)$. Therefore $f\in H$.

For the final assertion of the theorem, note that (b), (c), (d) and (e) are equivalent by Theorem \ref{TRES}. That (a) implies (c) is trivial. To show that (d) implies (a), note that if $Y$ is compact then $X$ has $\mathscr{P}$, as $X=\mathrm{supp}(h)$ for some $h\in H$ (see Remark \ref{KGFF}).
\end{proof}

\begin{remark}
Observe that in the proof of Theorem \ref{CDXF} we actually proved that $H$ is a closed subalgebra of $C_b(X)$.
\end{remark}

\begin{remark}
The set
\[C_\mathscr{P}(X)=\big\{f\in C_b(X):\mathrm{supp}(f)\mbox{ has }\mathscr{P}\big\}\]
where $X$ is a space and $\mathscr{P}$ is a topological property, has been also considered in \cite{AG} and \cite{T}. (See also \cite{AN}.) In \cite{AG} (Theorem 2.2) conditions are given which are necessary and sufficient for $C_\mathscr{P}(X)$ to be a Banach space. The approach in \cite{T} is quite algebraic.
\end{remark}

The next two theorems are to provide examples of topological properties $\mathscr{P}$ and $\mathscr{Q}$ satisfying the assumption of Theorem \ref{CDXF}. Specifically, in Theorem \ref{CTTF} we let $\mathscr{P}$ and $\mathscr{Q}$ be, respectively, separability and metrizability, and in Theorem \ref{UUP} we let $\mathscr{P}$ and $\mathscr{Q}$ be, respectively, the Lindel\"{o}f property and paracompactness. Among other things, we will show that in these cases (with the notation of Theorem \ref{CDXF}) $Y$ is countably compact, and is non-normal if $X$ is non-$\mathscr{P}$.

Theorem \ref{CTTF} is known (see \cite{Ko6}); we include it here (along with its proof) partly to demonstrate an application of Theorem \ref{CDXF}, and partly because of the existing duality between Theorem \ref{CTTF} and Theorem \ref{UUP}. Besides, the proof in Theorem \ref{CTTF} illustrates how our methods may be used to extract results not generally expected to be deducible from the standard Gelfand Theory. Let us recall the following.

A space $X$ is called \emph{countably compact} if every countable open cover of $X$ has a finite subcover, or equivalently, if every countable infinite subspace of $X$ has a limit point in $X$. It is well known that if $X$ is a locally compact Hausdorff space then $C_0(X)=C_{00}(X)$ if and only if every $\sigma$-compact subspace of $X$ is contained in a compact subspace of $X$ (see Problem 7G.2 of \cite{GJ}); in particular, $C_0(X)=C_{00}(X)$ implies that $X$ is countably compact.

Let $D$ be an uncountable discrete space. Denote by $D_\lambda$ the subspace of $\beta D$ consisting of elements in the closure in $\beta D$ of countable subsets of $D$. In \cite{W}, the author proves the existence of a continuous (2-valued) function $f:D_\lambda\backslash D\rightarrow[0,1]$ which is not continuously extendible to $\beta D\backslash D$. This, in particular, proves that $D_\lambda$ is not normal. (To see this, suppose the contrary. Note that $D_\lambda\backslash D$ is closed in $D_\lambda$, as $D$ is locally compact and then open in $\beta D$. By the Tietze--Urysohn Extension Theorem, $f$ is extendible to a continuous bounded function over $D_\lambda$, and therefore over $\beta D_\lambda=\beta D$; note that $D\subseteq D_\lambda$. But this is not possible.)

A theorem of Tarski guarantees for an infinite set $I$ the existence of a collection $\mathscr{I}$ of cardinality $|I|^{\aleph_0}$ consisting of countable infinite subsets of $I$ such that the intersection of any two distinct elements of $\mathscr{I}$ is finite. (See \cite{Ho}.) Note that the collection of all subsets of cardinality at most $\mathfrak{m}$ in a set of cardinality $\mathfrak{n}\geq\mathfrak{m}$ is of cardinality at most $\mathfrak{n}^\mathfrak{m}$.

Observe that if $X$ is a space and $D\subseteq X$, then
\[U\cap\mathrm{cl}_XD=\mathrm{cl}_X(U\cap D)\]
for any open-closed subspace $U$ of $X$.

\begin{theorem}\label{CTTF}
Let $X$ be a locally separable metrizable space. Then
\[C_s(X)=\big\{f\in C_b(X):\mathrm{supp}(f)\mbox{ is separable}\big\}\]
is a Banach algebra isometrically isomorphic to $C_0(Y)$ for some unique locally compact Hausdorff space $Y$, namely $Y=\lambda_{C_s(X)} X$. Furthermore,
\begin{itemize}
\item[\rm(1)] $Y$ is countably compact.
\item[\rm(2)] $Y$ is non-normal, if $X$ is non-separable.
\item[\rm(3)] $C_0(Y)=C_{00}(Y)$.
\item[\rm(4)] $\dim C_s(X)=d(X)^{\aleph_0}$.
\item[\rm(5)] The following are equivalent:
\begin{itemize}
\item[\rm(a)] $X$ is separable.
\item[\rm(b)] $C_s(X)$ is unital.
\item[\rm(c)] $C_s(X)$ contains $\mathbf{1}$.
\item[\rm(d)] $Y$ is compact.
\item[\rm(e)] $Y=\beta X$.
\end{itemize}
\end{itemize}
\end{theorem}

\begin{proof}
Let $\mathscr{P}$ be separability and $\mathscr{Q}$ be metrizability. Then the pair $\mathscr{P}$ and $\mathscr{Q}$ satisfies the assumption of Theorem \ref{CDXF}. (Observe that any subspace of a separable metrizable space is separable.) As remarked in Part 1.3 the space $X$ may be represented as a disjoint union
\[X=\bigcup_{i\in I}X_i,\]
where $I$ is an index set and $X_i$ is a non-empty separable open-closed subspace of $X$ for each $i\in I$. To simplify the notation, for any $J\subseteq I$ denote
\[H_J=\bigcup_{i\in J}X_i.\]
Observe that $H_J$ is open-closed in $X$, thus it has open-closed closure in $\beta X$. Also, as we will see now,
\begin{equation}\label{JHGK}
\lambda_{C_s(X)} X=\bigcup\{\mathrm{cl}_{\beta X}H_J:J\subseteq I\mbox{ is countable}\}.
\end{equation}
To show (\ref{JHGK}), let $f\in C_s(X)$. Then $\mathrm{supp}(f)$ is separable and is then Lindel\"{o}f. Therefore $\mathrm{supp}(f)\subseteq H_J$ for some countable $J\subseteq I$. Thus
\[\mathrm{cl}_{\beta X}\mathrm{Coz}(f)\subseteq\mathrm{cl}_{\beta X}H_J.\]
Now, let $J\subseteq I$ be countable. Let
\[g=\chi_{H_J}.\]
Then $g$ is continuous, as $H_J$ is open-closed in $X$, and $\mathrm{supp}(g)=H_J$ is separable. Since $\mathrm{cl}_{\beta X}H_J$ is open in $\beta X$ we have
\[\mathrm{cl}_{\beta X}H_J=\mathrm{int}_{\beta X}\mathrm{cl}_{\beta X}H_J=\mathrm{int}_{\beta X}\mathrm{cl}_{\beta X}\mathrm{Coz}(g)\subseteq\lambda_{C_s(X)} X.\]

Note that (5) follows from Theorem \ref{CDXF} and that (3) implies (1); as if (3) holds, then every countable infinite subspace of $Y$ (being $\sigma$-compact) is contained in a compact subspace of $Y$, and therefore has a limit point in $Y$. (See the remarks preceding the statement of the theorem.)

(3). We need to show that every $\sigma$-compact subspace of $\lambda_{C_s(X)} X$ is contained in a compact subspace of $\lambda_{C_s(X)} X$. Let $A$ be a $\sigma$-compact subspace of $\lambda_{C_s(X)} X$. Then
\[A=\bigcup_{n=1}^\infty A_n,\]
where $A_n$ is compact for each positive integer $n$. By compactness and the representation given in (\ref{JHGK}) we have
\begin{equation}\label{DJB}
A_n\subseteq\mathrm{cl}_{\beta X}H_{J_1}\cup\cdots\cup\mathrm{cl}_{\beta X}H_{J_{k_n}}
\end{equation}
for some countable $J_1,\ldots,J_{k_n}\subseteq I$. Let
\begin{equation}\label{FFB}
J=\bigcup_{n=1}^\infty(J_{k_1}\cup\cdots\cup J_{k_n}).
\end{equation}
Then $J$ is countable and $A\subseteq\mathrm{cl}_{\beta X}H_J$.

(2). Let $x_i\in X_i$ for each $i\in I$. Then
\[D=\{x_i:i\in I\}\]
is a closed discrete subspace of $X$, and since $X$ is non-separable, it is uncountable. Suppose to the contrary that $\lambda_{C_s(X)} X$ is normal. Then
\[\lambda_{C_s(X)} X\cap\mathrm{cl}_{\beta X}D=\bigcup\{\mathrm{cl}_{\beta X}H_J\cap\mathrm{cl}_{\beta X}D:J\subseteq I\mbox{ is countable}\}\]
is normal, as it is closed in $\lambda_{C_s(X)} X$. Now, let $J\subseteq I$ be countable. Since $\mathrm{cl}_{\beta X}H_J$ is open-closed in $\beta X$ (using the observation preceding the statement of the theorem) we have
\[\mathrm{cl}_{\beta X}H_J\cap\mathrm{cl}_{\beta X}D=\mathrm{cl}_{\beta X}(\mathrm{cl}_{\beta X}H_J\cap D)=\mathrm{cl}_{\beta X}(H_J\cap D)=\mathrm{cl}_{\beta X}\big(\{x_i:i\in J\}\big).\]
But $\mathrm{cl}_{\beta X}D=\beta D$, as $D$ is closed in (the normal space) $X$. Therefore
\[\mathrm{cl}_{\beta X}\big(\{x_i:i\in J\}\big)=\mathrm{cl}_{\beta X}\big(\{x_i:i\in J\}\big)\cap\mathrm{cl}_{\beta X}D=\mathrm{cl}_{\beta D}\big(\{x_i:i\in J\}\big).\]
Thus
\[\lambda_{C_s(X)} X\cap\mathrm{cl}_{\beta X}D=\bigcup\{\mathrm{cl}_{\beta D}E:E\subseteq D\mbox{ is countable}\}=D_\lambda,\]
contradicting the fact that $D_\lambda$ is not normal.

(4). Since $X$ is non-separable, $I$ is infinite and $d(X)=|I|$. The Tarski Theorem implies the existence of a collection $\mathscr{I}$ of cardinality $|I|^{\aleph_0}$ consisting of countable infinite subsets of $I$, such that the intersection of any two distinct elements of $\mathscr{I}$ is finite. Let $f_J=\chi_{H_J}$ for any $J\in\mathscr{I}$. No element in
\[\mathscr{F}=\{f_J:J\in\mathscr{I}\}\]
is a linear combination of other elements (as each element of $\mathscr{I}$ is infinite and each pair of distinct elements of $\mathscr{I}$ has finite intersection). Observe that $\mathscr{F}$ is of cardinality $|\mathscr{I}|$. Thus
\[\dim C_s(X)\geq|\mathscr{I}|=|I|^{\aleph_0}=d(X)^{\aleph_0}.\]
On the other hand, if $f\in C_s(X)$, then $\mathrm{supp}(f)$ is Lindel\"{o}f (as it is separable) and thus $\mathrm{supp}(f)\subseteq H_J$, where $J\subseteq I$ is countable; therefore, we may assume that $f\in C_b(H_J)$. Conversely, if $J\subseteq I$ is countable, then each element of $C_b(H_J)$ can be extended trivially to an element of $C_s(X)$ (by defining it to be identically $0$ elsewhere). Thus $C_s(X)$ may be viewed as the union of all $C_b(H_J)$, where $J$ runs over all countable subsets of $I$. Note that if $J\subseteq I$ is countable, then $H_J$ is separable; thus any element of $C_b(H_J)$ is determined by its value on a countable set. This implies that for each countable $J\subseteq I$, the set $C_b(H_J)$ is of cardinality at most $\mathfrak{c}^{\aleph_0}=2^{\aleph_0}$. There are at most $|I|^{\aleph_0}$ countable $J\subseteq I$. Therefore
\begin{eqnarray*}
\dim C_s(X)\leq\big|C_s(X)\big|&\leq& \Big|\bigcup\big\{C_b(H_J):J\subseteq I\mbox{ is countable}\big\}\Big|\\&\leq& 2^{\aleph_0}\cdot|I|^{\aleph_0}=|I|^{\aleph_0}=d(X)^{\aleph_0}.
\end{eqnarray*}
\end{proof}

\begin{theorem}\label{UUP}
Let $X$ be a locally Lindel\"{o}f paracompact space. Then
\[C_l(X)=\big\{f\in C_b(X):\mathrm{supp}(f)\mbox{ is Lindel\"{o}f}\big\}\]
is a Banach algebra isometrically isomorphic to $C_0(Y)$ for some unique locally compact Hausdorff space $Y$, namely $Y=\lambda_{C_l(X)} X$. Furthermore, the following are equivalent:
\begin{itemize}
\item[\rm(a)] $X$ is Lindel\"{o}f.
\item[\rm(b)] $C_l(X)$ is unital.
\item[\rm(c)] $C_l(X)$ contains $\mathbf{1}$.
\item[\rm(d)] $Y$ is compact.
\item[\rm(e)] $Y=\beta X$.
\end{itemize}
If $X$ is moreover locally compact then
\begin{itemize}
\item[\rm(1)] $Y$ is countably compact.
\item[\rm(2)] $Y$ is non-normal, if $X$ is non-Lindel\"{o}f.
\item[\rm(3)] $C_0(Y)=C_{00}(Y)$.
\item[\rm(4)] $\dim C_l(X)=\ell(X)^{\aleph_0}$.
\end{itemize}
\end{theorem}

\begin{proof}
Let $\mathscr{P}$ be the Lindel\"{o}f property and $\mathscr{Q}$ be paracompactness. We need to know that the pair $\mathscr{P}$ and $\mathscr{Q}$ satisfies the assumption of Theorem \ref{CDXF}. It is well known that $\mathscr{Q}$ is hereditary with respect to closed subspaces. (See Corollary 5.1.29 of \cite{E}.) It is obvious that $\mathscr{P}$ and $\mathscr{Q}$ satisfy conditions (1) and (3) of Theorem \ref{CDXF}. It is also known that any paracompact space with a dense Lindel\"{o}f subspace is Lindel\"{o}f (see Theorem 5.1.25 of \cite{E}), that is, $\mathscr{P}$ and $\mathscr{Q}$ satisfy conditions (2) of Theorem \ref{CDXF}. The first part of the theorem together with the equivalence of conditions (a)--(c) now follows from Theorem \ref{CDXF}. The proofs for (1)--(4) are analogous to the ones we have given for the corresponding parts of Theorem \ref{CTTF}. (One needs to assume a representation for $X$ as given in Part 1.4.)
\end{proof}

\section*{Acknowledgement}

The author wishes to thank the anonymous referee for reading the manuscript, the prompt report given almost within a month, and comments which motivated the author to make numerous improvements.

\bibliographystyle{amsplain}

\begin{thebibliography}{10}

\bibitem{AG} S.K. Acharyya and S.K. Ghosh, Functions in $C(X)$ with support lying on a class of subsets of $X$. \textit{Topology Proc.} \textbf{35} (2010), 127-–148.

\bibitem{AN} S. Afrooz and M. Namdari, $C_\infty(X)$ and related ideals. \textit{Real Anal. Exchange} \textbf{36} (2010), 45--54.

\bibitem{A} A.R. Aliabad, F. Azarpanah and M. Namdari, Rings of continuous functions vanishing at infinity. \textit{Comment. Math. Univ. Carolin.} \textbf{45} (2004),
    519-–533.

\bibitem{Be} E. Behrends, $M$-structure and the Banach--Stone Theorem. Springer, Berlin, 1979.

\bibitem{E} R. Engelking, General Topology. Second edition. Heldermann Verlag, Berlin, 1989.

\bibitem{GJ} L. Gillman and M. Jerison, Rings of Continuous Functions. Springer--Verlag, New York--Heidelberg, 1976.

\bibitem{HR} E. Hewitt and K.A. Ross, Abstract Harmonic Analysis I. Springer--Verlag, New York, 1979.

\bibitem{Ho} R.E. Hodel, Jr., Cardinal functions I, in: K. Kunen and J.E. Vaughan (Eds.), Handbook of Set-theoretic Topology, Elsevier, Amsterdam, 1984, pp. 1--61.

\bibitem{Ko3} M.R. Koushesh, Compactification-like extensions. \textit{Dissertationes Math. (Rozprawy Mat.)} \textbf{476} (2011), 88 pp.

\bibitem{Ko4} M.R. Koushesh, The partially ordered set of one-point extensions. \textit{Topology Appl.} \textbf{158} (2011), 509--532.

\bibitem{Ko5} M.R. Koushesh, A pseudocompactification. \textit{Topology Appl.} \textbf{158} (2011), 2191--2197.

\bibitem{Ko6} M.R. Koushesh, The Banach algebra of continuous bounded functions with separable support. \textit{Studia Math.} \textbf{210} (2012), 227--237.

\bibitem{Ko13} M.R. Koushesh, One-point extensions and local topological properties. \textit{Bull. Aust. Math. Soc.} (5 pp.), to appear. \verb"arXiv:1210.8074 [math.GN]"

\bibitem{Ko11} M.R. Koushesh, Topological extensions with compact remainder. \textit{J. Math. Soc. Japan} (47 pp.), to appear.

\bibitem{Ko10} M.R. Koushesh, Representations theorems for normed algebras. \textit{J. Aust. Math. Soc.} (22 pp.), to appear. \verb"arXiv:1204.6660 [math.FA]"

\bibitem{Ko12} M.R. Koushesh, Continuous mappings with null support. (40 pp.), in preparation.  \verb"arXiv:1302.2235 [math.FA]"

\bibitem{T} A. Taherifar, Some generalizations and unifications of $C_{00}(X)$, $C_\psi(X)$ and $C_0(X)$. (13 pp.) \verb"arXiv:1210.6521 [math.GN]"

\bibitem{W} N.M. Warren, Properties of Stone--\v{C}ech compactifications of discrete spaces. \textit{Proc. Amer. Math. Soc.} \textbf{33} (1972), 599-–606.

\end{thebibliography}

\end{document}